\documentclass[11pt]{article}
\usepackage{natbib}
\usepackage{amssymb}
\usepackage{amsthm}
\usepackage{latexsym}
\usepackage{amsmath,amsfonts,amssymb}
\usepackage{epsfig}
\usepackage{fancyhdr}
\usepackage{sectsty}
\usepackage{chngpage}

\newfont{\SectionTitle}{cmcsc18}
\newtheorem{theorem}{Theorem}
\newtheorem{lemma}{Lemma}
\newcommand{\E}{{\rm ~E\,}}

\begin{document}
\title{Developments in maximum likelihood unit root tests}
\author{Ying Zhang\\
Department of Mathematics and Statistics\\
Acadia University\\
\\
Hao Yu and A. Ian McLeod \\
Department of Statistical and Actuarial Sciences \\
University of Western Ontario}
\date{December 19, 2011}
\bigskip
\maketitle
\newpage
\begin{abstract}
The exact maximum likelihood estimate (MLE) provides a test statistic for the unit root
test that is more powerful \citep[p. 577]{Fuller96} than the usual least squares approach.
In this paper a new derivation is given for the asymptotic distribution of this
test statistic that is simpler and more direct than the previous method.
The response surface regression method is used to obtain a fast algorithm that
computes accurate finite-sample critical values.
This algorithm is available in the R package {\tt mleur} that is
available on CRAN.
The empirical power of the new test is shown to be much better than the usual
test not only in the normal case but also for innovations generated
from an infinite variance stable distribution as well as for innovations
generated from a GARCH$(1,1)$ process.
\end{abstract}

\newpage

\section{Introduction}\label{SectionIntro}

The AR$(1)$ model is widely used in many applications as well as in unit root testing.
Modern approaches to the unit root testing problem emphasize the importance of model
selection \citep{PfaffBook2006, Enders2010, Patterson2010}.
This paper focuses on testing the null model known as random walk,
\begin{equation}
\nabla z_t =  a_t, \quad t=1,2,...,
\label{NullModel}
\end{equation}
where $\nabla z_t = z_t - z_{t-1}$ and $a_t$ are independent and normally
distributed with mean zero and variance $\sigma^2_a$.
The alternative is assumed to be the stationary AR$(1)$ model with intercept term $\beta$,
\begin{equation}
z_t = \beta + \rho z_{t-1} +a_t, \quad t=1,2,...,
\label{ARONE}
\end{equation}
where $|\rho| < 1$.

Sometimes it is assumed that $\beta=0$ is {\bf known.}  This case corresponds to the zero-mean AR$(1)$ processes.
Both of these models were discussed in the original formulation of the unit root testing problem
by \citet{DickeyFuller79} but using the least-squares estimates (LSE) instead of the MLE.
The random walk model and the stationary AR$(1)$ alternative provide
a suitable family of models for many financial and economic time series.
However, as is discussed in \S \ref{Section:Application}, other methods are needed
if the diagnostic checks reveal that further lagged values need to be included
in the model.

\citet[p. 577]{Fuller96} indicates that if the objective is to test the hypothesis of a unit root against the alternative of a stationary process with an unknown mean, the test statistics associated with the exact MLE are more powerful than that with the LSE.
The exact MLE referred to is the MLE in the stationary case that corresponds to the alternative hypothesis in the
unit root test.
Empirical power comparisons among various unit root tests showed that the MLE based tests had much higher power
than the Dickey-Fuller (DF) tests \citep{Pantula1994}.
Extensions of the MLE method to the ARMA$(1,1)$ and other ARMA processes were discussed by \citet{Shin1998}.

\citet[\S 10.1.3]{Fuller96} and \citet{GF99}
derive the limiting distributions of normalized statistics associated with the exact MLE
unit root test under eqns. (\ref{NullModel}) and (\ref{ARONE}).
This approach is indirect whereas our new derivation in \S \ref{SectionComputer} is essentially simpler
and more direct.
Our method using the Taylor series linearization of the test statistic
is carried out through symbolic computer algebra.
The exact MLE itself is also derived symbolically through the
solution of a cubic equation in \S 2.
The usual approach to the exact MLE using a numerical optimization technique can occasionally have convergence problems.
This more direct approach using a symbolic Taylor series linearization
is easier to generalize to other problems as well.
It is known that computer algebra may handle complicated statistical inference problems \citep{AndStaf}.
There are several examples in time series analysis.
\citet{SmithField} show how a symbolic operator can be used to calculate the joint cumulants of the linear combinations of products of discrete Fourier transforms.
\citet{ZhangMcLeod06} discuss a computer algebra approach to the asymptotic bias and variance coefficients to order $O(1/n)$ for linear estimators in stationary time series.
Computer algebra no doubt has many more applications in statistics and time series analysis.

In \S \ref{Section:Implementation}, using response surface curves, we show that the critical values for
the MLE test may be efficiently computed.
With our fast algorithm, in \S \ref{Section:PowerComparison}, we demonstrate that the exact MLE test provides
not only a sizeable increase in power but also the robustness against alternative specifications for the innovations
such as an infinite variance stable distribution and a GARCH$(1,1)$ process.
We illustrate how to implement the exact MLE unit root test with two real world examples in \S \ref{Section:Application}.

\section{Exact MLE}\label{SectionMLE}

The AR$(1)$ model (\ref{ARONE}) may also be written,
\begin{equation}
z_t -\mu= \rho (z_{t-1}-\mu) +a_t, t=1,2,...,
\label{ARONEMU}
\end{equation}
where $E(z_t)=\mu$ and $\beta = \mu(1-\rho)$.
When $\mu$ is known, without loss of generality, it is assumed that $\mu=0$.
The time series process is stationary if $|\rho|<1$.
In the random walk case $\rho=1$ and the process is said to be unit root non-stationary.
If $\rho>1$, the process is explosively non-stationary.

Most of unit root tests have been derived under the data generation model,
\begin{equation}
z_t-\mu=\rho (z_{t-1}-\mu)+a_t, t=1,2,...,
\label{DGMOne}
\end{equation}
where $z_0 = \mu$ is a fixed value.
The only difference between model (\ref{ARONEMU}) and (\ref{DGMOne}) is the initial value.
The time series represented by (\ref{DGMOne}) is mostly same as that by (\ref{ARONEMU}), except that under (\ref{DGMOne}) the process is asymptotically stationary when $|\rho| <1$ and the LSE is the maximum likelihood estimator of $\rho$ conditionally on the initial value.

First consider the zero-mean stationary time series under (\ref{ARONEMU}).
Its initial value follows a normally distributed random variable with zero mean and a variance of $\sigma^2/(1-\rho^2)$.
The exact {\bf log-likelihood} function of $n$ consecutive observations, $z_t, t=1,...,n$, may be written as \citep{MinAzz93}
\begin{equation}
l(\sigma^2, \rho) = -{n\over 2}\log(2\pi)-{n\over 2}
\log(\sigma^2)+{1\over 2} \log(1-\rho^2)-{1\over 2\sigma^2}(a-2\rho
b+\rho^2c)
\label{ExactLikelihood}
\end{equation}
where
\begin{equation}
a = \sum \limits_{t=1}^n {z_t}^2,\ \
b = \sum\limits_{t=2}^nz_tz_{t-1},\ \
c = \sum\limits_{t=2}^{n-1}{z_t}^2.
\label{ABC}
\end{equation}
Maximizing $l(\sigma^2, \rho)$ in eqn.(\ref{ExactLikelihood}), \citet{White61}, and \citet{MinAzz93} show that the exact
MLE of $\rho$ is the unique real root of the following equation, whose absolute value is less than one.
\begin{equation}
{{n-1}\over n} c \rho^3-{{n-2}\over n} b \rho^2-(c+{a\over n})\rho+b=0.
\label{CubicMLE}
\end{equation}
\citet{DentMin78}, \citet{Hasza80}, and \citet{MinAzz93} point out that the exact MLE may be written as
\begin{equation}
\hat {\rho}=2 ({{{d_2}^2-3d_1}\over 9})^{1/2} \cos ({\theta \over
3}+{4\pi \over 3})-{d_2\over 3},
\label{RhoHat}
\end{equation}
where
$$
\theta={\cos^{-1}\{{{9d_2d_1-27d_0-2{d_2}^3}\over {2({d_2}^2-3d_1)^{3/2}}}}\},
$$
and
$$
d_2 = -{{n-2}\over {n-1}}{b\over c},\ \
d_1 = -{n\over {n-1}}(1+{a\over{nc}}),\ \
d_0 = {n\over {n-1}} {b\over c}.
$$
Using {\it Mathematica\/} \citep{Math}, the cubic equation (\ref{CubicMLE}) is easily solved
and the exact MLE $\hat \rho$ may be expressed as the ratio of complex polynomials,
\begin{eqnarray*}
\hat\rho=& &{(n-2)b}/3c (-1+n)+((1-i\sqrt{3})\\
& &(-b^2(-2+n)^2+3c(-1+n)(-a-c n)))/(32^{2/3}c(1-n)\\
& &(16b^3-18abc-24b^3n+27abcn+9bc^2n+12b^3n^2-9abcn^2\\
& &-27bc^2n^2-2b^3n^3+18bc^2n^3+(((16b^3-18abc-24b^3n+27abcn\\
& &+9bc^2n+12b^3n^2-9abcn^2-27bc^2n^2-2b^3n^3+18bc^2n^3)^2+4(-b^2(-2+n)^2\\
& &+3c(-1+n)(-a-c n))^3))^{1/2})^{1/3})-1/(6 2^{1/3}c(1-n))\\
& &((1+i \sqrt{3})(16b^3-18abc-24b^3n+27abcn+9bc^2n+12b^3n^2\\
& &-9abcn^2-27bc^2n^2-2b^3n^3+18bc^2n^3+(((16b^3-18abc-\\
& &-24b^3n+27abcn+9bc^2n+12b^3n^2-9abcn^2-27bc^2n^2-2b^3n^3\\
& &+18bc^2n^3)^2+4(-b^2(-2+n)^2+3c(-1+n)(-a-c n))^3))^{1/2})^{1/3})
\end{eqnarray*}
where $i=\sqrt {-1}$, and $a$, $b$ and $c$ are defined in (\ref{ABC}).

For a stationary AR(1) process with an unknown mean under (\ref{ARONEMU}), there are two mean correction methods: sample mean correction and the maximum likelihood mean estimation.
It is well known that for ARMA$(p,q)$ model, the sample mean is asymptotically efficient \citep[\S 7.1]{Brock87}.
The exact MLE for the $\mu$ may be obtained iteratively as in \citet{McLeod2008} but in the AR$(1)$ case
the sample mean has close to 100\% efficiency in finite samples \cite[Table 3]{McLeod2008}.
For speed and convenience we may just consider the sample mean estimator in eqn. (\ref{ARONEMU}).
That is, the exact MLE is the $\hat\rho$ described above with $z_t-{\overline z}_n$ ($t=1,...,n$) replacing $z_t$ where ${\overline z}_n$ is the sample mean, which is denoted as $\hat\rho_\mu$.

Under the stationary alternative, the exact MLE and the LSE have the same limiting distribution
\citep[\S 8]{Brock87} but this is not the case under the non-stationary null hypothesis eqn. (\ref{NullModel}).
The next section provides a new derivation of this distribution.

\section{Computer Algebra Derivations to Limiting Distributions}\label{SectionComputer}

In the unit root case, $\rho =1$, we consider the random walk
\begin{equation}
z_t= z_{t-1} +a_t, \quad t=1,2,...,
\label{RW}
\end{equation}
where $\{a_t\}$ is a sequence of IID random variables with mean 0 and finite variance $\sigma_a^2>0$.

Fixed $z_0=0$, the random walk process may be generated by
\begin{eqnarray}
z_{t}=\sum\limits_{j=1}^t{a_j}.
\end{eqnarray}

For the zero-mean case, the normalized and pivotal type statistics may be written as
\begin{equation}
\hat{\delta} = n(\hat \rho -1)
\label{MLEdelta}
\end{equation}
where $\hat \rho$ is described in \S \ref{SectionMLE}, and
\begin{equation}
\hat{\tau} = {1\over{\hat\sigma}}(\sum\limits_{t=2}^n{z^2_{t-1}})^{1/2} (\hat\rho-1)
\label{MLEtau}
\end{equation}
where
$$
\hat\sigma^2= (n-2)^{-1} \sum\limits_{t=2}^n (z_t-\hat\rho z_{t-1})^2.
$$

For the unknown mean case, the normalized statistic may be written as
\begin{equation}
\hat{\delta}_\mu = n(\hat\rho_\mu-1)
\label{MLEdeltaMu}
\end{equation}
where $\hat\rho_\mu$ is described in \S \ref{SectionMLE}, and the corresponding pivotal statistic may be written by
\begin{equation}
\hat{\tau}_\mu = \hat\sigma_\mu^{-1} (\sum\limits_{t=2}^n{(z_{t-1}-{\overline z}_n)^2})^{1/2} (\hat\rho_\mu-1),
\label{MLEtauMu}
\end{equation}
where
$$\hat\sigma_\mu^2= (n-3)^{-1}\sum\limits_{t=2}^n (z_t-{\overline z}_n-\hat\rho_\mu (z_{t-1}-{\overline z}_n))^2.$$
The limiting distributions of statistics in eqns. (\ref{MLEdelta}), (\ref{MLEtau}), (\ref{MLEdeltaMu}) (\ref{MLEtauMu})
are given in Theorems 1 and 2 below.

\begin{theorem}\label{TheoremOne}
Under a random walk (\ref{RW}),
\begin{equation}
n(\hat \rho_\mu-1) \stackrel{D} \longrightarrow \frac 12
\left(\mathfrak{C}_\mu-\sqrt{\mathfrak{C}_\mu^2-4\mathfrak{C}_\mu+2\mathfrak{B}_\mu}\right),
\label{LimitOne}
\end{equation}

\begin{equation}
\hat{\tau}_\mu \stackrel{D} \longrightarrow \frac {\sqrt{\mathfrak{A}_\mu}}
{2} \left(\mathfrak{C}_\mu-\sqrt{\mathfrak{C}_\mu^2-4\mathfrak{C}_\mu+2\mathfrak{B}_\mu}\right),
\label{LimitTwo}
\end{equation}
where
$$\mathfrak{A}_\mu=\int_0^1 W^2(t)\,dt-\left(\int_0^1W(t)\,dt\right)^2,$$
$$\mathfrak{B}_\mu=\mathfrak{A}_\mu^{-1}\left(\left(\int_0^1W(t)\,dt\right)^2+\left(W(1)-\int_0^1W(t)\,dt\right)^2 \right),$$
$$\mathfrak{C}_\mu=\mathfrak{A}_\mu^{-1}\left(\frac 12 \left(W^2(1)-1\right) -W(1)\int_0^1 W(t)\,dt+\left(\int_0^1W(t)\,dt\right)^2\right),$$
and $\{W(t),\ 0 \leq t \leq 1\}$ is a standard Wiener process.
\end{theorem}

\begin{theorem}\label{TheoremTwo}
Under a random walk (\ref{RW}),
\begin{equation}
n(\hat \rho-1) \stackrel{D} \longrightarrow \frac 12
\left(\mathfrak{C}-\sqrt{\mathfrak{C}^2-4\mathfrak{C}+2\mathfrak{B}}\right),
\label{LimitThree}
\end{equation}

\begin{equation}
\hat{\tau} \stackrel{D} \longrightarrow \frac {\sqrt{\mathfrak{A}}}
{2} \left(\mathfrak{C}-\sqrt{\mathfrak{C}^2-4\mathfrak{C}+2\mathfrak{B}}\right),
\label{LimitFour}
\end{equation}
where
$$
\mathfrak{A}=\int_0^1 W(t)^2\,dt,
$$
$$
\mathfrak{B}=\mathfrak{A}^{-1}W(1)^2,
$$
$$
\mathfrak{C}=\mathfrak{A}^{-1}{\frac 12
\left(W^2(1)-1\right)}
$$
and $\{W(t),\ 0 \leq t \leq 1\}$ is a standard Wiener process.
\end{theorem}

To fix ideas, below is demonstrated how {\it Mathematica} helping to prove the limiting distribution of $n(\hat \rho_\mu-1)$,
eqn. (\ref{LimitOne}) in Theorem \ref{TheoremOne} .

$\hat\rho_\mu$ may be further simplified as follows:
\begin{eqnarray*}
\hat\rho_\mu = -\frac{(n-2)G}{3(1-n)} + \frac{(1-i \sqrt{3})u} {2^{2 \over
3}3 (1-n) (v+\sqrt{v^2+4u^3})^{1\over3}}
-{{(1+i\sqrt{3})(v+\sqrt{v^2+4u^3})^{1\over3}}\over
{2^{1\over 3}6}(1-n)}
\end{eqnarray*}
where
\begin{eqnarray*}
u=&& -(n-2)^2 G^2+3 (1-n) (H+n),\\
v=&&16G^3 - 18 G H - 24 G^3 n + 27 G H n + 9 G n\\
    &&+ 12 G^3 n^2- 9 G H n^2 - 27 G n^2 - 2 G^3 n^3 +18 G n^3
\end{eqnarray*}
where
\begin{equation}
G=b/c, \ \ H=a/c
\label{GH}
\end{equation}
where $a$, $b$ and $c$ are defined in (\ref{ABC}) with $z_t-{\overline z}_n$ ($t=1,...,n$) replacing $z_t$.
The limiting distributions of $G$ and $H$ are given in the following lemma.

\begin{lemma} \label{Lemma}
Under a random walk (\ref{RW}),
\begin{equation}
(c/(\sigma^2 n^2), n(H-1), n(G-1)) \stackrel{D} \longrightarrow (\mathfrak{A}_\mu, \mathfrak{B}_\mu, \mathfrak{C}_\mu)
\end{equation}
where $\mathfrak{A}_\mu$, $\mathfrak{B}_\mu$, $\mathfrak{C}_\mu$ are defined in Theorem \ref{TheoremOne}.
\end{lemma}

A detailed proof of Lemma \ref{Lemma} can be found in Appendix \ref{App}.

\begin{proof}[Proof of eqn. (\ref{LimitOne}) in Theorem \ref{TheoremOne}]
Let $W = n(G-1)$ and $X =n(H-1)$.
Lemma \ref{Lemma} implies that $ W =O_p(1)$ and $ X =O_p(1)$.
$\hat {\rho}_\mu$ can be considered as a function of $1/n$ with $1+ W/n$ and $1+X/n$ replacing $G$ and $H$.
In order to obtain the limit distribution of $\hat {\rho}_\mu$, taking $\hat {\rho}_\mu$ with one-term Taylor expansion with respect to $1/n$ at zero,
\begin{equation}
\hat\rho_\mu=1+\frac{1}{2n}\left(W-\sqrt{W^2-4W+2X}\right)+ \frac{1}{n^2} R_n(W, X).
\label{Taylor}
\end{equation}
where $\sup_{n \geq 1} |R_n(W,X)| \leq C(W,X)$  that is a continuous function of $W$ and $X$.
Below is a {\it Mathematica} script and its output for deriving eqn (\ref{Taylor}).
\begin{verbatim}
In[1]:= u = -(n-2)^2G^2+3(1-n)(H+n);
In[2]:= v = 16 G^3-18 G H - 24 G^3 n+27 G H n + 9G n + 12 G^3 n^2
            -9 G H n^2-27 G n^2-2 G^3 n^3+18 G n^3;
In[3]:= rho = -(-2+n)G/(3(1-n))
        +((1-i Sqrt[3])(u))/(3 2^(2/3)(1-n)(v+Sqrt[v^2+4u^3])^(1/3))
        -1/(6 2^(1/3)(1-n))((1+i Sqrt[3])(v+Sqrt[v^2+4u^3])^(1/3));
In[4]:= G = 1 + W/n; H = 1 + X/n; n = 1/z;
        Simplify[Series[rho, {z, 0, 1}]]
\end{verbatim}
and the output of the final input is
\begin{verbatim}
Out[4]= 1+1/2(W+i(4W-W^2-2X)^(1/2))z+O[z]^2
\end{verbatim}
which leads to eqn. (\ref{Taylor}).
Following the fact that $W=O_p(1)$, $X=O_p(1)$ and the continuity of $C(W, X)$,
$$\frac{1}{n} R_n(W, X) \stackrel{P} \longrightarrow 0.$$
By Lemma 1,
$$
(W, X) \stackrel{D}  \longrightarrow (\mathfrak{C}_\mu, \mathfrak{B}_\mu).
$$
Thus applied the continuous mapping theorem described in Appendix \ref{App} and the Slutsky's theorem to eqn. (\ref{Taylor}), enq. (\ref{LimitOne}) is obtained.
\end{proof}

The limiting distributions of $\hat\tau_\mu$, $\hat {\rho}$, and $\hat\tau$ in eqns. (\ref{LimitTwo}) - (\ref{LimitFour}) can be very similarly derived.

\citet[Theorem 10.1.10 and Corollary 10.1.10]{Fuller96} show that eqn. (\ref{LimitOne}) and eqn. (\ref{LimitThree}) hold,
which indicates that the computer algebra derivations implemented here are appropriate.
Other than the normalized statistics, eqns. (\ref{LimitTwo}) and (\ref{LimitFour}) show the limiting distributions on the unit root boundary of pivotal statistics for both zero-mean and unknown mean cases.

\section{Methods of Implementing the Test}\label{Section:Implementation}

The asymptotic distribution may be evaluated by computer simulation
methods for Brownian motion.
Such methods are discussed in the book by \citet{Iacus2008}.
Then this asymptotic distribution could be used to obtain
critical values and/or p-values for the test.
As we will show below, this method will not work unless the
series length is very long.

The simplest approach is to use a Monte-Carlo test.
Under general conditions this approach provides an accurate test
that can be efficiently computed using parallel processing capabilities
found on many modern computer environments.
For example, the necessary steps are outlined below for the normalized test:\\

\noindent 1) simulate $M$ random walks under (\ref{NullModel}) with the length of $n$ and compute the simulated testing statistic sample,
$n({\hat \rho}^1_\mu-1)$, $n({\hat \rho}^2_\mu-1)$, ..., $n({\hat\rho}^M_\mu-1)$;\\
2) compute the observed testing statistic value for the given time series $\{z_t\}$, $n({\hat\rho}^0_\mu-1)$;\\
3) count the number of times $k$ that the simulated test statistic $n(\hat {\rho}^i_\mu-1)$ ($i=1, ..., M$)
is less than or equal to the observed test statistic $n({\hat\rho}^0_\mu-1)$;\\
4) compute the p-value as $(k+1)/(M+1)$.\\

\noindent Instead of using independent normal random variables to generate the random walks in Step 1),
we could use a bootstrap sample of the residuals. This test has been implemented in the function {\tt mctest} in our R package for MLE unit root tests \citep{mleur}.

An even more computationally efficient approach is to use response
surface regression \citep{MacKinnon02} to estimate the quantile functions for the exact
distribution.
The response surface regressions are of the form,
$$Q^{\alpha}(n)=\theta_{\infty}+{\theta_1}/n+{\theta_2}/n^2+{\theta_3}/n^3+\epsilon,$$
where $Q^{\alpha}(n)$ is an $\alpha$ percentile of the finite sample distribution that is estimated
using $N$ replications and $\epsilon$ is an error term.
The curve was fit with the massive cluster computer SHARCNET utilizing 221 compute nodes for about
ten hours.
Thirty-six series lengths $n$ used were 20, (5), 100, (20), 300, (50), 500, (100), 1000.
For each series length $N=200000$ replications were done and this as repeated $M=100$ times.
From this the mean and variance of each percentile were estimated and used in
a weighted least squares regression to obtain the final fitted regression.
The weighted least squares approach is needed to account for heteroskedasticity in the error terms.

In the case of the model specified in eqns. (\ref{NullModel}) and (\ref{ARONE}),
the critical values for the test statistic $\hat{\tau}_\mu$ given in eqn. (\ref{MLEtauMu}) are:
\begin{equation}\label{RSR}
\hat Q^{\alpha}(n) = \left\{ \begin{array}{rl}
 -3.110 - 4.652/n - 51.466/n^2 &\mbox{1\%  point} \\
 -2.531 - 2.062/n - 17.529/n^2 &\mbox{5\%  point} \\
 -2.233 - 1.219/n - 8.178/n^2  &\mbox{10\%  point}
       \end{array} \right.
\end{equation}
Figure 1 below illustrates these curves for series lengths up to 500.
The dashed line shows the critical point from the asymptotic distribution.
It is seen that a reasonably large sample is needed to obtain
accurate critical values using the asymptotic distribution.
The y-axis each panel is scaled so scaling unit is the same.
This scaling reveals the critical values corresponding 10\% converge more
quickly while the 1\% critical values converge slowly.

Extensive simulation experiments were performed for a variety of series lengths, $n$,
and parameters, $\rho$, to check that the p-values produced by the Monte-Carlo method
agreed with that produced by the critical values from eqn. (\ref{RSR}).

Implementing the explicit expression of the exact MLE derived in \S \ref{SectionMLE} and the critical value equations such as eqns. (\ref{RSR}), our R function {\tt mleur} for the MLE unit root tests is available in our R package {\tt mleur} \citep{mleur}.

\section{Power Comparisons}\label{Section:PowerComparison}

We investigated the power of the MLE unit root tests under various types of innovations in comparison with that of the standard Dickey-Fuller test.
Under our null model (\ref{NullModel}), and alternative model (\ref{ARONE}) or (\ref{ARONEMU}), the unknown mean case is more realistic than the known mean case.
Thus the MLE unit root test was implemented with the sample mean correction in the normalized form $n(\hat\rho_\mu-1)$ or the pivotal form $\hat\tau_\mu$, denoted by MLEn or MLEp respectively.
In R the standard Dickey-Fuller test is implemented in several packages and usually the pivotal form of test statistic is used.
We used the implementation of the Dickey-Fuller pivotal test for the same model as (\ref{NullModel}) and (\ref{ARONE}) with an unknown mean or interpret in
the R package {\tt urca} by \cite{urca}, represented by DF in this paper.
The function {\tt GetPower} for making such power comparisons is given in our package \citep{mleur}

In constructing critical value eqns. (\ref{RSR}), the simulated series were
assumed to be Gaussian.
But since the asymptotic distribution only relies on the assumption that the innovations are
independent with mean zero and finite variance $\sigma_a^2$,
it is plausible that the critical values given in eqns. (\ref{RSR})
may also be applicable for other non-normal distributions with finite variance.
In fact, using our R function {\tt GetPower}, we found no difference from the
normal distribution results with Student-t on 5 degrees of freedom.
A more challenging question is how well these results continue to hold
when these assumptions are not met as in the case of infinite variance
distributions, or series exhibiting conditional heteroscedasticity and non-linear
dependence.
To answer this question, a portion of our simulation results is shown in Table 1.
$25,000$ replications were done for series of lengths $n=30, 70, 100, 200$ and parameters
$\rho = 0.65, 0.85, 0.9, 0.95, 1.0$ for the innovations generated by a stable distribution and a GARCH model described in the following.
With so many replications the 95\% margin-of-error (MOE) was about 0.0062
or 0.62 in percentage terms.
This computations took less than 3 hours on a multicore PC.

The random variable $Z$ has a stable distribution with index $\alpha$, scale $\sigma>0$, skewness $|\beta|<1$
and location $\mu \in {\cal R}$ if its characteristic function is given by,
\begin{equation}
{\E}(e^{i t Z })=\left\{
\begin{array}{cl}
 \exp \left\{ -\sigma |t|^{\alpha}\, \left( 1-i \beta \ \ \mbox{sgn}(t)\, \tan \frac{\pi \alpha }{2}\right) +\,i \mu  t\right\} & \mbox{if}\
  \alpha \neq 1 \nonumber\\\\
 \exp \left\{ -\sigma  |t|\, \left( 1+i \beta  \frac{2}{\pi}\ \ \mbox{sgn}(t)\, \log |t| \right) +\,i \mu  t\right\}
 & \mbox{if}\  \alpha =1,
\end{array}
\right.
\newcounter{cf}
\setcounter{cf}{\value{equation}}
\end{equation}
where
\[
\mbox{sgn}(t)=\left\{
\begin{array}{rl}
 1  &\mbox{if}\quad t >0 \\
 0  &\mbox{if}\quad t=0\\
 -1 &\mbox{if}\quad t<0.
\end{array}
\right.
\]
Since it has been suggested that many financial time series appear to have
a stable distribution with $\alpha$ in the range $(1.35, 1.75)$,
$\alpha$ was set to $1.5$ for our simulations.
Also, $\sigma=1$, $\beta=0$, and $\mu=0$.

A GARCH($1, 1$) sequence $a_t, t = \ldots, -1, 0, 1, \ldots  $ is of the form
$$
 a_{t} = \sigma_t \epsilon_{t}
$$
and
$$\sigma_t^2 =
  \omega +  \alpha_1 a^2_{t-1} +  \beta_1 \sigma_{t-1}^2,
$$
where we took $\epsilon_t$ to be independent standard normal,
$\omega=10^{-6}, \alpha_1 = 0.2$ and $\beta_1 = 0.7$.
The parameters were chosen to approximate models that have
been used in actual applications.

Table 1 shows that there can be substantial difference in power
between the MLE unit root test and the Dickey-Fuller tests not only in the normal case but also for innovations generated from an infinite variance stable distribution as well as for innovations generated from a GARCH(1,1) process.
It is observed that the size of the test is slightly inflated for the non-normal case,
so this needs to be taken into account in the power comparison.
In general it appears that the pivotal form of the test statistic,
MLEp, is preferable to the normalized form, MLEn.
MLEp is just as robust as MLEn and has slightly better power.

Further empirical power analysis may easily be carried out similarly with our R function {\tt GetPower}.

\section{Illustrative Applications}\label{Section:Application}

In actual applications, it is recommended that diagnostic checks be done
for residual autocorrelation.
If there is significant autocorrelation in the residuals of the
fitted AR(1) model then other methods such as the augmented Dickey-Fuller
test must be used.
The model building procedure needed for this Dickey-Fuller
test family is discussed
by \citet{PfaffBook2006} and is available in the R package {\tt urca} \citep{urca}.
Our R package {\tt mleur} \citep{mleur} provides suitable model
diagnostic checks for applying the MLE root test and is
used in the applications discussed below.
R scripts to generate the analyses reported below are available
in our package documentation.

\subsection{Velocity of money}\label{Subsection:Money}

The time series plot for the velocity of money in the U.S.
1869-1970 is shown in Figure \ref{velTsplot}.
From the plot we see the series has historically exhibited a strong
stochastic trends characteristic of random walk behavior.
No doubt with modern emphasis on fiscal policies to control inflation
the series has stabilized.
But just for a numerical illustration of the difference in the unit
root tests we will compare the maximum likelihood and least squares or
Dickey-Fuller tests.
The first step in the analysis is the check that the fitted model
is adequate and that no additional lags are required.
Figure \ref{velDiag} shows the diagnostic checks
for this data.
The residuals appear non-normal but in view of the simulation
results this is not a concern.
Most importantly no evidence of residual autocorrelation is found in the
fitted AR(1) model.
Applied the unit root tests, the pivotal test statistics for
the MLE and DF tests were respectively $-0.26$ and $-3.28$.
The MLE test is not even close to being significant at the 10\% level
while the DF test has a p-value between 5\% and 1\%.
The MLE unit root test gives a result that appears to be more
in line with the overall impression of strong stochastic
trends exhibited in Figure \ref{velTsplot}.
Even though the length of the series was $102$, there is
a considerable difference in the conclusion between the two methods.

\subsection{Bond yield differences}\label{Subsection:Bonds}

The annual difference in Mood's BAA and AAA corporate bond yields
from 1976 to 2010 is shown in Figure \ref{DiffBATsplot}.
From the diagnostic check plots,
we conclude that there is no significant autocorrelation in the residuals
and so the AR(1) may be fit.
The DF test is not significant at 10\% whereas the MLE test does
reject at the 10\% level.
This is not surprising in view of the empirical power computations.

\section{Summary}\label{Section:Summary}

In this paper we presented a new derivation of the asymptotic distribution for the MLE
unit root test utilizing computer algebra to obtain an explicit expression for the
MLE and a Taylor series linearization for the test statistic.
This technique is {\bf no} doubt applicable in other situations where the manual derivation
is difficult.

An efficient computational method based on the response surface curves has been
implemented to obtain critical values of the MLE test statistics.
An empirical power study has demonstrated that not only does the MLE procedure
outperform the LSE in the Gaussian case but also for fat-tailed distributions, infinite variance distributions,
and for weak dependence as exhibited in a GARCH$(1,1)$ process.
The R package {\tt mleur} based on the developments in this paper is available on CRAN.

Two illustrative applications of the test demonstrate that unit root testing
also requires diagnostic checking.
It is important for proper applications that there be no residual autocorrelation
present in the fitted AR$(1)$ model.\\

\noindent ACKNOWLEDGEMENTS

The authors would like to thank the Editor for his encouragement and also two referees for their helpful suggestions and comments.
The authors research was supported by Natural Sciences and Engineering Research Council of Canada (NSERC).

\newpage
\section{Appendix}\label{App}
First we state the Donsker's theorem \citep{Billingsley99}.
Let $\{a_t\}$ be a sequence of IID random variables with mean 0 and finite
variance $\sigma^2>0$, $z_t=\sum_{j=1}^t a_j$ and $z_0=0$. Then
$$
\left\{ \frac{z_{[nt]}} {\sqrt{n}\sigma},\ 0\leq t\leq 1\right\}
\stackrel{D} \longrightarrow \{W(t),\ 0 \leq t\leq 1\}
$$
in the Skorokhod space $D[0,\ 1]$ with $J_1$ topology, where $[x]$
denotes the integer part of $x$. One of the important applications
of the Donker's theorem is the following continuous mapping theorem. If
$f(\cdot)$ is a continuous functional on $[0,\ 1]$, then
$$
f \left(\frac{z_{[nt]}}{\sqrt{n}\sigma}\right)\stackrel{D}
\longrightarrow f(W(t)).
$$
\begin{proof}[Proof of Lemma \ref{Lemma}]
We have
$$
\frac{\bar{z}_n}{\sqrt{n}}= \int_0^1\frac{z_{[nt]}}{\sqrt{n}}\,
 dt+\frac{z_n}{n \sqrt{n}}.
 $$
By the Donsker's theorem
$$
\frac{\bar{z}_n}{\sqrt{n}} \stackrel{D} \longrightarrow \sigma
\int_0^1 W(t)\,dt.
$$
Similarly,
$$
\frac 1 {n^2} \sum_{t=1}^{n-1}
z_t^2=\int_0^1\left(\frac{z_{[nt]}}{\sqrt{n}}\right)^2\,dt
\stackrel{D} \longrightarrow  \sigma^2\int_0^1 W^2(t)\,dt.
$$
It can be shown that
\begin{equation}
c=\sum_{t=2}^{n-1}\left(z_t-\bar{z}_n\right)^2=\sum_{t=1}^{n-1}z_t^2-z_1^2-(n+2)\bar{z}_n^2+2\bar{z}_n (z_1+z_n).
\label{A1}
\end{equation}
By some simple algebra steps,
\begin{eqnarray*}
  b-c &=& \sum_{t=1}^{n-1}(z_{t+1}-z_t)(z_t-\bar{z}_n) +(z_1-\bar{z}_n)^2\\
   &=&\frac 12 z_n^2-\frac 12
   \sum_{t=1}^na_t^2-\bar{z}_n\sum_{t=2}^n  a_t +(z_1-\bar{z}_n)^2,
\end{eqnarray*}
that is,
\begin{equation}
\frac{b-c}{n} = \frac 12 \left(\frac{z_n}{\sqrt{n}}\right)^2-\frac
1{2n}\sum_{t=1}^na_t^2-\frac{\bar{z}_n}{\sqrt{n}}\frac{z_n}{\sqrt{n}} +
\frac{\bar{z}_n}{\sqrt{n}}\frac{a_1}{\sqrt{n}}+\left(\frac{z_1}{\sqrt{n}}
-\frac{\bar{z}_n}{\sqrt{n}}\right)^2.
\label{A2}
\end{equation}
Moreover
\begin{equation}
\frac{a-c}{n}=\left(\frac{z_1}{\sqrt{n}}-\frac{\bar{z}_n}{\sqrt{n}}\right)^2
+\left(\frac{z_n}{\sqrt{n}}-\frac{\bar{z}_n}{\sqrt{n}}\right)^2.
\label{A3}
\end{equation}
Since
$$
\frac{b-c}{n} = n(G-1) (\frac{c}{n^2})
$$
and
$$
\frac{a-c}{n} = n(H-1) (\frac{c}{n^2}),
$$
where $G$ and $H$ are defined in eqn (\ref{GH}).  Eqns. (\ref{A1}), (\ref{A2}), and (\ref{A3}) imply that
$$
(\frac{c}{\sigma^2 n^2}, n(H-1), n(G-1)) = f(\frac{z_{[nt]}}{\sqrt{n}\sigma})+o_p(1),
$$
where $f$ is a functional.  Hence by the continuous mapping theorem described above and Slutsky's theorem,
$$
(\frac{c}{\sigma^2 n^2}, n(H-1), n(G-1)) \stackrel{D}\longrightarrow  f(W(t)).
$$
Applied the Cramer-Rao device, we find the marginal distributions as
$$
 \frac{c}{\sigma^2 n^2}\stackrel{D}\longrightarrow \int_0^1
 W^2(t)\,dt-\left(\int_0^1W(t)\,dt\right)^2=\mathfrak{A}_\mu,
$$
$$
n(G-1)
\stackrel{D}\longrightarrow \mathfrak{A}_\mu^{-1}\left(\frac 12
\left(W^2(1)-1\right) -W(1)\int_0^1
W(t)\,dt+\left(\int_0^1W(t)\,dt\right)^2\right)=\mathfrak{C}_\mu,
$$
and
$$
n(H-1)\stackrel{D}\longrightarrow \mathfrak{A}_\mu^{-1}\left(\left(\int_0^1W(t)\,dt\right)^2+\left(W(1)-
\int_0^1W(t)\,dt\right)^2 \right) = \mathfrak{B}_\mu.
$$
\end{proof}
\newpage

\begin{table}[ht!]
\begin{center}
\begin{adjustwidth}{-0.25in}{-1in}
\begin{tabular}{|rr|rrr|rrr|rrr|}
  \hline
\multicolumn{2}{c}{   } & \multicolumn{3}{c}{normal} & \multicolumn{3}{c}{stable} & \multicolumn{3}{c}{GARCH(1,1)} \\
  \hline
\hline
$n$ & $\rho$ & DF & MLEn & MLEp & DF & MLEn & MLEp & DF & MLEn & MLEp \\[0.7ex]
\hline
30	&	0.65	&	39.8	&	56.5	&	59.6	&	36.5	&	55.7	&	59.4	&	 42.0	&	56.2	&	 59.0 \\
70	&	0.65	&	97.6	&	99.8	&	99.7	&	97.7	&	98.6	&	98.1	&	 95.6	&	98.9	&	 98.8 \\
100	&	0.65	&	100.0	&	100.0	&	100.0	&	99.7	&	99.3	&	99.1	&	 99.7	&	100.0	&	 99.9 \\
200	&	0.65	&	100.0	&	100.0	&	100.0	&	99.9	&	99.7	&	99.6	&	 100.0	&	100.0	&	 100.0 \\ [0.5ex]
30	&	0.85	&	12.1	&	16.6	&	18.3	&	11.6	&	12.5	&	13.6	&	 14.1	&	18.3	&	 20.0 \\
70	&	0.85	&	37.4	&	55.1	&	57.4	&	33.6	&	53.9	&	57.0	&	 39.7	&	55.9	&	 57.8 \\
100	&	0.85	&	63.2	&	83.2	&	84.2	&	65.2	&	84.3	&	84.6	&	 64.5	&	81.4	&	 82.1 \\
200	&	0.85	&	99.6	&	100.0	&	100.0	&	99.4	&	98.9	&	98.4	&	 98.7	&	99.7	&	 99.6 \\ [0.5ex]
30	&	0.90	&	8.4	&	10.7	&	11.9	&	8.9	&	8.3	&	8.9	&	9.8	&	11.8	 &	13.1 \\
70	&	0.90	&	19.4	&	29.8	&	31.4	&	17.4	&	24.6	&	26.8	&	 22.0	&	32.0	&	 33.7 \\
100	&	0.90	&	33.3	&	51.0	&	52.8	&	29.7	&	49.4	&	52.8	&	 36.3	&	51.9	&	 53.5 \\
200	&	0.90	&	86.8	&	97.2	&	97.0	&	89.3	&	95.7	&	94.8	&	 84.3	&	94.7	&	 94.7 \\ [0.5ex]
30	&	0.95	&	6.7	&	7.6	&	8.3	&	6.5	&	5.5	&	5.7	&	7.8	&	8.1	&	9.1 \\
70	&	0.95	&	9.2	&	12.5	&	13.3	&	9.0	&	9.5	&	9.9	&	11.0	&	 14.7	&	15.4 \\
100	&	0.95	&	12.5	&	19.0	&	19.8	&	11.8	&	14.2	&	15.1	&	 14.6	&	21.2	&	 22.0 \\
200	&	0.95	&	32.5	&	51.1	&	52.5	&	28.9	&	47.7	&	50.7	&	 36.0	&	52.6	&	 53.9 \\ [0.5ex]
30	&	1.00	&	5.5	&	4.9	&	5.5	&	6.3	&	4.3	&	4.4	&	7.0	&	6.1	&	6.7 \\
70	&	1.00	&	5.2	&	5.1	&	5.3	&	6.0	&	3.7	&	3.7	&	7.0	&	6.1	&	6.4 \\
100	&	1.00	&	5.0	&	5.3	&	5.6	&	6.0	&	3.9	&	3.8	&	6.7	&	6.3	&	6.3 \\
200	&	1.00	&	4.9	&	4.8	&	4.9	&	5.9	&	3.8	&	3.8	&	6.2	&	6.2	&	6.3 \\
 \hline
 \end{tabular}
\end{adjustwidth}
\caption{
Empirical power of 5\% unit root tests based on 25,000 simulations using innovations
from various distributions.
Tests were Dickey-Fuller (DF), MLE normalized (MLEn), and MLE pivotal (MLEp).
The distributions used were standard normal,
stable distribution with index parameter 1.5, and a GARCH$(1,1)$
process.
The 0.95 level MOE for the table percentage is 0.62}
\end{center}
\end{table}

\newpage

\begin{figure}[ht!]
\begin{center}
\includegraphics*[scale=0.8]{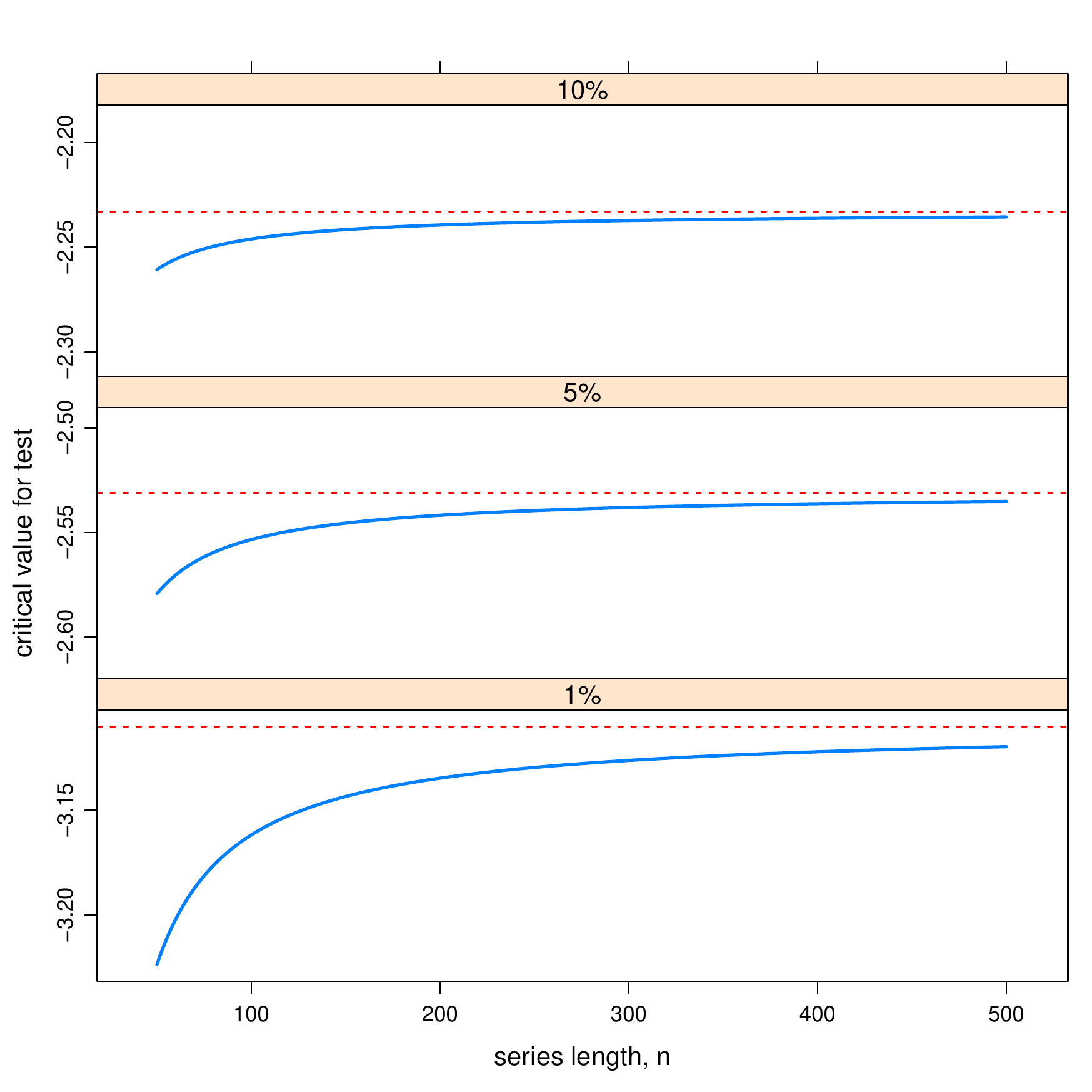}
\end{center}
\caption{The 1\%, 5\% and 10\% critical values for the MLE test statistic $\hat{\tau}_\mu$}
\label{Fig.VAR1}
\end{figure}

\newpage

\begin{figure}[ht!]
\begin{center}
\includegraphics*[scale=0.8]{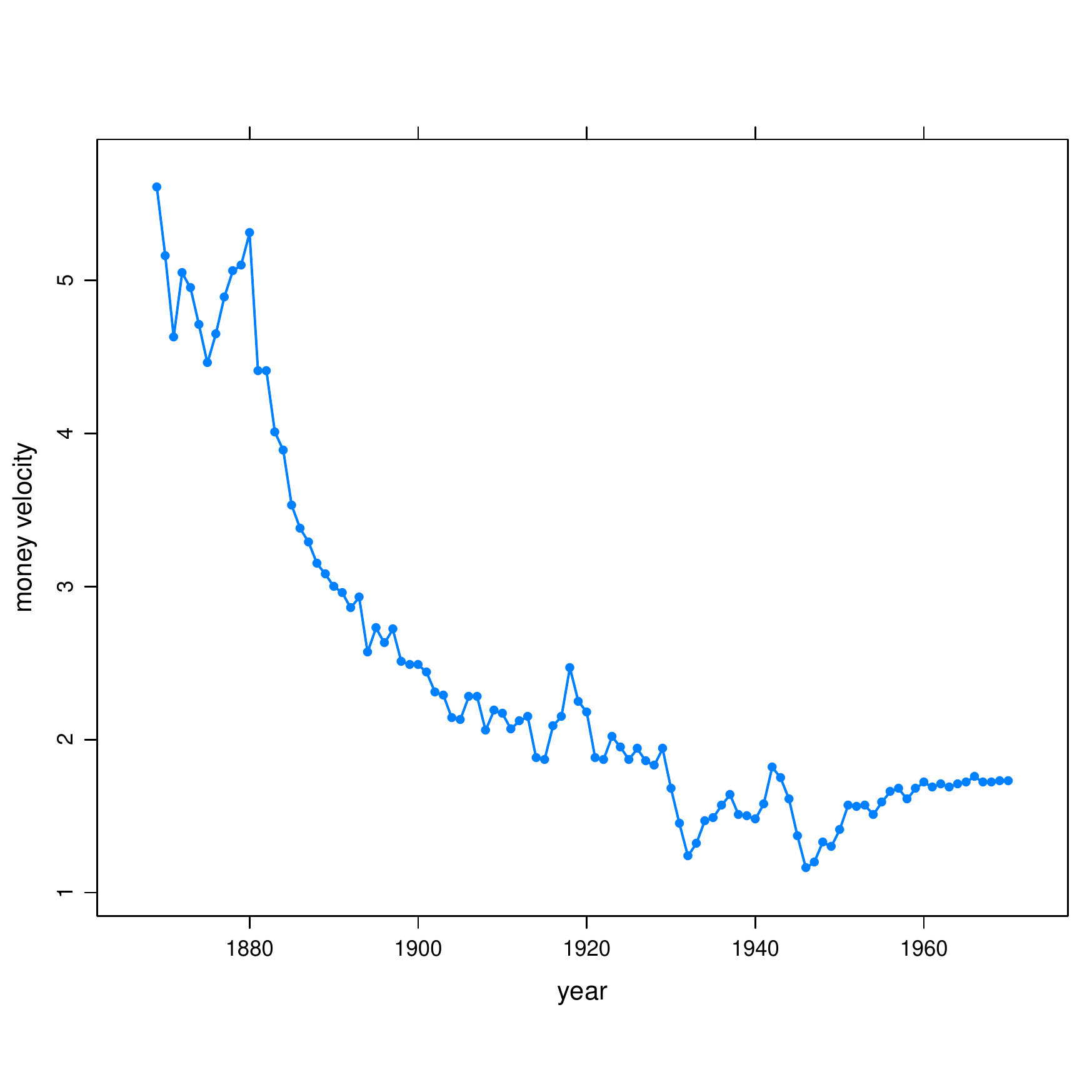}
\end{center}
\caption{Time series plot for money velocity}
\label{velTsplot}
\end{figure}

\newpage

\begin{figure}[ht!]
\begin{center}
\includegraphics*[scale=0.8]{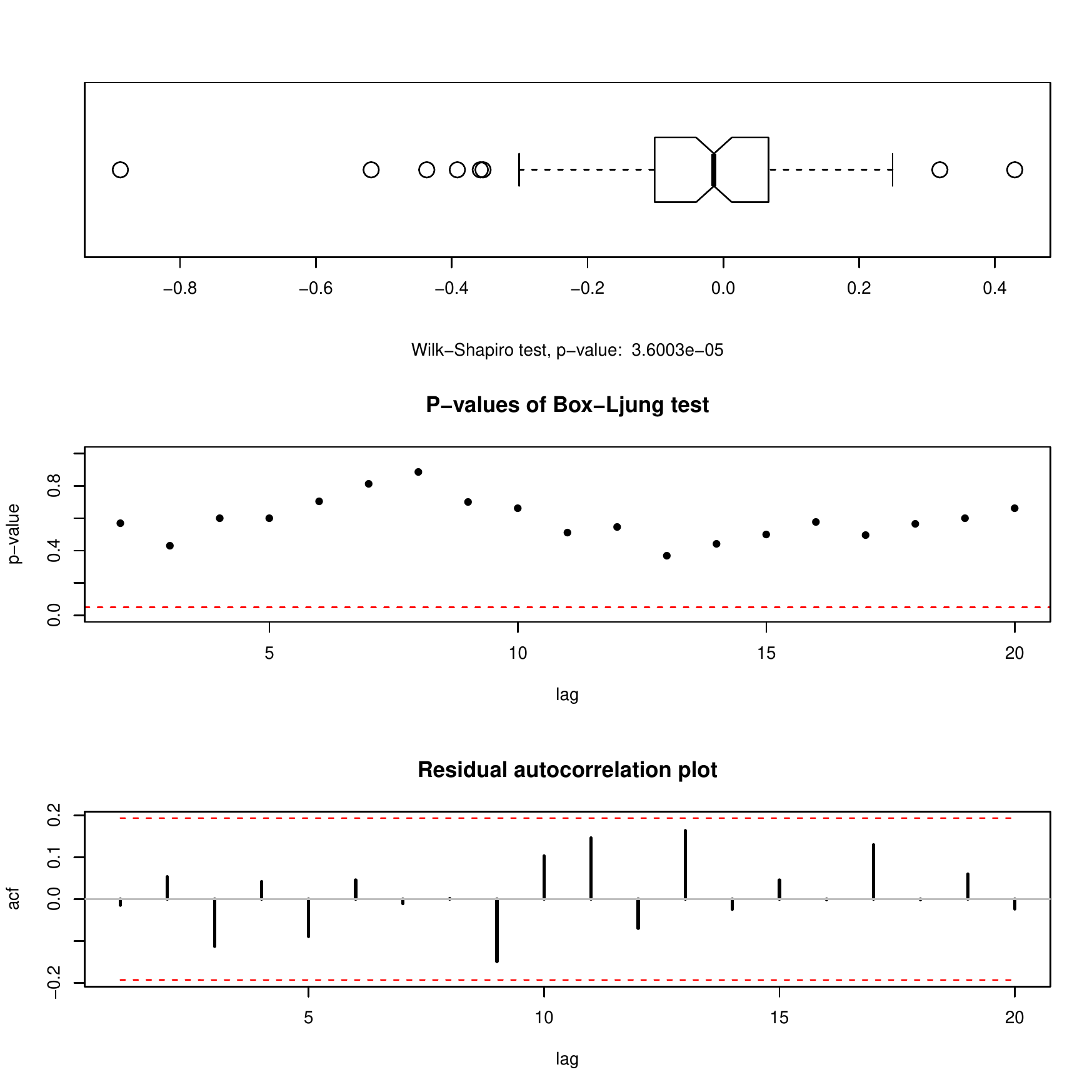}
\end{center}
\caption{Diagnostic plots for money velocity time series}
\label{velDiag}
\end{figure}

\newpage

\begin{figure}[ht!]
\begin{center}
\includegraphics*[scale=0.8]{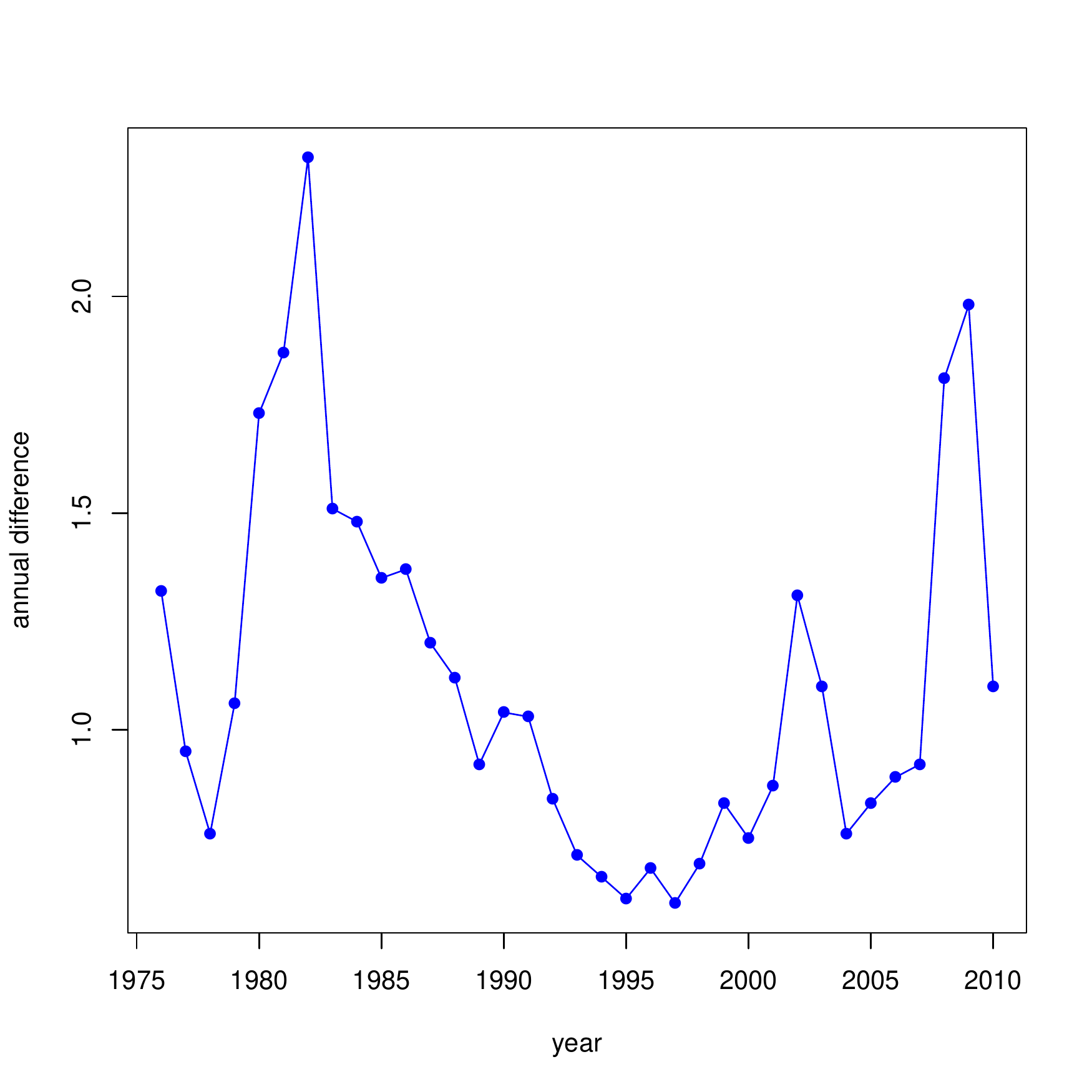}
\end{center}
\caption{Time series plot of difference in bond yields}
\label{DiffBATsplot}
\end{figure}

\newpage

\bibliographystyle{chicagoa}
\bibliography{Unitroot}
\end{document}